\documentclass[11pt]{amsart}

\usepackage{amsmath,amsthm,amsfonts,amssymb,bm,graphicx,cprotect,enumerate}

\usepackage{microtype}

\usepackage{fullpage}
\usepackage{hyperref}

\hypersetup{colorlinks=true, citecolor=blue, linkcolor=blue, urlcolor=blue,
pdfstartview=FitH, pdfauthor=Balázs Maga and Péter Maga, pdftitle=Random power series near the endpoint of the convergence interval}

\newcommand{\NN}{\mathbf{N}}
\newcommand{\QQ}{\mathbf{Q}}
\newcommand{\RR}{\mathbf{R}}

\newcommand{\eps}{\varepsilon}

\renewcommand{\leq}{\leqslant}
\renewcommand{\geq}{\geqslant}

\DeclareMathOperator{\p}{P}

\DeclareMathOperator{\pp}{\mathbf{p}}
\DeclareMathOperator{\dom}{dom}

\newtheorem{proposition}{Proposition}
\newtheorem{theorem}{Theorem}
\newtheorem{corollary}{Corollary}
\newtheorem{lemma}{Lemma}

\theoremstyle{remark}
\newtheorem{remark}{Remark}

\begin{document}

\author{Bal\'azs Maga}
\author{P\'eter Maga}

\address{E\"otv\"os Lor\'and University, P\'azm\'any P\'eter s\'et\'any 1/C, Budapest, H-1117 Hungary}
\email{magab@cs.elte.hu}
\address{MTA Alfr\'ed R\'enyi Institute of Mathematics, POB 127, Budapest H-1364, Hungary}
\email{magapeter@gmail.com}
\address{MTA R\'enyi Int\'ezet Lend\"ulet Automorphic Research Group}

\title{Random power series near the endpoint of the convergence interval}

\thanks{The first author was supported by the \'Uj Nemzeti Kiv\'al\'os\'ag Program grant \'UNKP-17-2 of the Hungarian Ministry of Human Capacities. The second author was supported by the Postdoctoral Fellowship of the Hungarian Academy of Sciences and by NKFIH (National Research, Development and Innovation Office) grants NK~104183 and ERC\underline{\phantom{ }}HU\underline{\phantom{ }}15~118946.}

\begin{abstract}
In this paper, we are going to consider power series
\begin{displaymath}
\sum_{n=1}^{\infty} a_nx^n,
\end{displaymath}
where the coefficients $a_n$ are chosen independently at random from a finite set with uniform distribution. We prove that if the expected value of the coefficients is positive (resp. negative), then
\begin{displaymath}
\lim_{x\to 1-}\sum_{n=1}^{\infty} a_nx^n=\infty\qquad
(\text{resp. }\lim_{x\to 1-}\sum_{n=1}^{\infty} a_nx^n=-\infty)
\end{displaymath} with probability $1$. Also, if the expected value of the coefficients is $0$, then
\begin{displaymath}
\limsup_{x\to 1-}\sum_{n=1}^{\infty} a_nx^n=\infty,\qquad \liminf_{x\to 1-}\sum_{n=1}^{\infty} a_nx^n=-\infty
\end{displaymath}
with probability $1$. We investigate the analogous question in terms of Baire categories.
\end{abstract}

\maketitle

\section{Introduction}

In complex analysis, the behaviour of random power series near the radius of convergence has been thorougly examined, partly due to the following classical problem: if we consider the Taylor series $f(z)=\sum_{n=0}^{\infty}a_nz^n$, what properties of the sequence $(a_n)_{n=0}^{\infty}$ imply that $f$ has its radius of convergence as a natural boundary, that is all of the points on its radius of convergence are singular? It turned out that random power series form a large family of such functions: it was proven in \cite{S} that if $f$ has a finite radius of convergence and $(\mathcal{A}_n)_{n=0}^{\infty}$ are independent, identically distributed random variables with uniform distribution on $\{|z|=1\}$, then for almost every choice, $f$ has a natural boundary on the radius of convergence. Later, somewhat stronger and more specific results were obtained, even in the recent years (see e.g. \cite{BS}).

These theorems showed that random power series in the complex plane tend to behave rather chaotically near the radius of convergence. In this paper, we investigate a similar question on the real line, motivated by a problem raised in \cite{KPP}. Although the results are somewhat natural and are easy to formulate, we did not manage to find them in the literature.

Let $D=\{d_1,\ldots,d_k\}$ be a finite set of real numbers. Then we may consider the random power series with coefficients from $D$, i.e.
\begin{displaymath}
f(x)=\sum_{n=1}^{\infty} a_nx^n,
\end{displaymath}
where each $a_n$ equals $d_j$ (for $1\leq j\leq k$) with probability $1/k$, independently in $n$. To exclude trivialities, assume from now on that $k\geq 2$.\footnote{By a slight change of notation, we start the power series with the order $1$ term, in order to index the random variables by $\NN$.}

To make this more rigorous, define for each $n\in\NN$, the probability space $(D,\mathcal{A}_n,\p_n)$, where $D$ is the fixed set above, $\mathcal{A}_n$ is the discrete topology on $D$, and for each $D'\subseteq D$,
\begin{displaymath}
\p_n(D')=\# D'/ \# D = \# D'/k
\end{displaymath}
with $\#$ standing for the cardinality.

Then set $(\Omega,\mathcal{A},\p)$ for the product probability space, i.e. $\Omega=\prod_{n\in \NN} D$, $\mathcal{A}$ is the set of Borel sets of $\Omega$ (in the product topology $\prod_{n\in\NN} \mathcal{A}_n$), $\p=\prod_{n\in\NN}\p_n$.

We will denote a general element of $\Omega$ by $(a_n)$, and by $a_n$ its $n$th coordinate (i.e. $a_n\in D$, $(a_n)\in\Omega$). To any $(a_n)\in \Omega$, we may associate the power series
\begin{displaymath}
f_{(a_n)}(x)=\sum_{n=1}^{\infty} a_nx^n.
\end{displaymath}
In most cases below, there will be a single sequence $(a_n)$ and a resulting power series $f_{(a_n)}$, therefore we write simply $f$ in place of $f_{(a_n)}$. Of course, when there is any chance for confusion, we return to the longer (and less loose) notation.

It is easy to see that the convergence radius of $f(x)$ is $1$ for almost all coefficient sequences $(a_n)$ (except for the trivial case $D=\{0\}$ which is already excluded by our assumption $k\geq 2$). In this paper, we investigate the behaviour of $f$, as $x$ tends to $1$ from below. It will turn out that the most important properties are the following:
\begin{equation}\label{eq_1}
\lim_{x\to 1-} f(x)=\infty,
\end{equation}
\begin{equation}\label{eq_2}
\lim_{x\to 1-} f(x)=-\infty,
\end{equation}
\begin{equation}\label{eq_3}
\limsup_{x\to 1-} f(x)=\infty\qquad \text{and} \qquad \liminf_{x\to 1-} f(x)=-\infty.
\end{equation}

Our first result is that one of these properties hold for almost all sequences.

\begin{proposition}\label{proposition}
We have
\begin{displaymath}
\p(\text{$f$ satisfies \eqref{eq_1} or \eqref{eq_2} or \eqref{eq_3}})=1.
\end{displaymath}
\end{proposition}

Moreover, we will identify which one of the three properties holds almost surely. We formulate this in two statements, depending on whether the expected value of a single coefficient vanishes or not.

\begin{theorem}\label{theorem_1} If $\sum_{d\in D} d> 0$, then
\begin{displaymath}
\p(\text{$f$ satisfies \eqref{eq_1}})=1.
\end{displaymath}
If $\sum_{d\in D} d < 0$, then
\begin{displaymath}
\p(\text{$f$ satisfies \eqref{eq_2}})=1.
\end{displaymath}
\end{theorem}

\begin{theorem}\label{theorem_2}
 If $\sum_{d\in D} d = 0$, then
\begin{displaymath}
\p(\text{$f$ satisfies \eqref{eq_3}})=1.
\end{displaymath}
\end{theorem}

In Section~\ref{sec:residuality}, we investigate the same properties of generic power series in the Baire categorial sense (see \cite[pp. 40-41]{O}). The corresponding statements are summarized as follows.

\begin{theorem}\label{theorem_3} If each element of $D$ is nonnegative (resp. nonpositive), then
\begin{displaymath}
\{(a_n)\in\Omega: \text{$f$ satisfies \eqref{eq_1}}\}\qquad \text{(resp. $\{(a_n)\in\Omega: \text{$f$ satisfies \eqref{eq_2}}\}$)}
\end{displaymath}
is residual.

If $D$ contains positive and negative elements simultaneously, then
\begin{displaymath}
\{(a_n)\in\Omega: \text{$f$ satisfies \eqref{eq_3}}\}
\end{displaymath}
is residual.
\end{theorem}

Now Theorem~\ref{theorem_2} and Theorem~\ref{theorem_3} have the following simple consequence via Bolzano's theorem on continuous functions, answering a question in \cite{KPP}.

\begin{corollary}\label{corollary}
If $\# D\geq 2$, and $\sum_{d\in D} d = 0$, then for almost all and residually many sequences $(a_n)\in\Omega$, the following holds. For any real number $y$, there are infinitely many numbers $0<x<1$ satisfying
\begin{displaymath}
y=\sum_{n=1}^{\infty} a_nx^n.
\end{displaymath}
\end{corollary}

For the sake of completeness, before starting the main investigations of the paper, we make it clear that properties \eqref{eq_1}-\eqref{eq_3} indeed define $\p$-measurable sets, that is, it makes sense to speak about the probabilities in Proposition~\ref{proposition} and Theorems~\ref{theorem_1}-\ref{theorem_2}. The argument in the proof of Lemma~\ref{lemma_1} is highly standard, so the experienced reader may skip it.

\begin{lemma}\label{lemma_1} For any $a\in\RR$,
\begin{displaymath}
\begin{split}
& \left\{(a_n)\in\Omega:\limsup_{x\to 1-} \sum_{n=1}^{\infty}a_nx^n > a\right\}\in\mathcal{A}, \left\{(a_n)\in\Omega:\liminf_{x\to 1-} \sum_{n=1}^{\infty}a_nx^n > a\right\}\in\mathcal{A}, \\
& \left\{(a_n)\in\Omega:\limsup_{x\to 1-} \sum_{n=1}^{\infty}a_nx^n < a\right\}\in\mathcal{A}, \left\{(a_n)\in\Omega:\liminf_{x\to 1-} \sum_{n=1}^{\infty}a_nx^n < a\right\}\in\mathcal{A}.
\end{split}
\end{displaymath}
\end{lemma}
\begin{proof}
We prove only the first statement, the remaining three ones follow similarly. Set $B_c=(c,\infty)$ for any $c\in\RR$.

First fix any $0\leq x<1$, and consider $g_x:\Omega\to\RR$ defined as $g_x((a_n))=\sum_{n=1}^{\infty} a_nx^n$. It is easy to see that $g_x$ is continuous: if $(a_n)$ and $\eps>0$ are given, then choose $N\in\NN$ such that $\max\{|d_1|,\ldots,|d_k|\}x^N/(1-x)<\eps/2$; we see that if we modify $(a_n)$ only in coordinates $n>N$, then $g_x((a_n))$ changes by less than $\eps$. Therefore, $g_x^{-1}(B_c)\in\mathcal{A}$ for any $c\in\RR$.

Now fix any $0\leq y<z<1$, and consider $g_{y,z}:\Omega\to\RR$ defined as $g_{y,z}((a_n))=\max_{y \leq x \leq z} g_x((a_n))$ (this maximum exists, as $\sum_{n=1}^{\infty}a_nx^n$ is continuous in $x\in[y,z]$). Using once again the continuity of $\sum_{n=1}^{\infty}a_nx^n$ in $x\in[y,z]$, we see, for any $c\in\RR$,
\begin{displaymath}
g_{y,z}^{-1}(B_c)=\bigcup_{x\in [y,z]\cap\QQ} g_x^{-1}(B_c) \in \mathcal{A},
\end{displaymath}
since each $g_{x}^{-1}(B_c)\in\mathcal{A}$.

Finally, observe that
\begin{displaymath}
\left\{(a_n)\in\Omega:\limsup_{x\to 1-} \sum_{n=1}^{\infty}a_nx^n > a\right\}= \bigcup_{j=1}^{\infty} \bigcap_{m=1}^{\infty} \bigcup_{l=m+1}^{\infty} g_{1-1/m,1-1/l}^{-1}(B_{a+1/j})\in\mathcal{A},
\end{displaymath}
since each $g_{1-1/m,1-1/l}^{-1}(B_{a+1/j})\in\mathcal{A}$.
\end{proof}

This lemma shows that the functions $\limsup_{x\to 1-} f(x)$ and $\liminf_{x\to 1-} f(x)$ are random variables. From this, it is clear that the properties \eqref{eq_1}-\eqref{eq_3} give rise to $\p$-measurable sets, e.g.
\begin{displaymath}
\left\{(a_n)\in\Omega:\text{$f$ satisfies \eqref{eq_1}}\right\}=\bigcap_{N=1}^{\infty}\left\{(a_n)\in\Omega:\liminf_{x\to 1-} \sum_{n=1}^{\infty}a_nx^n>N\right\}.
\end{displaymath}

\section{Extreme behaviour}

The goal of this section is to prove Proposition~\ref{proposition}, following the guiding principle that as $x$ tends to $1$ from below, our power series gets less and less sensitive to what its first few coefficients are. First of all, define the following Borel measures on $\RR$ (to see that they are Borel measures, recall Lemma~\ref{lemma_1}):
\begin{displaymath}
\mu_+(B)=\p\left(\limsup_{x\to 1-} f(x) \in B\right),\qquad \mu_-(B)=\p\left(\liminf_{x\to 1-} f(x) \in B\right).
\end{displaymath}
In other words, these are the distributions of $\limsup_{x\to 1-} f(x)$ and $\liminf_{x\to 1-} f(x)$, in particular, both of them are finite. One may easily see that Proposition~\ref{proposition} is equivalent to the fact that both $\mu_+$ and $\mu_-$ are the constant $0$ measures on $\RR$.

We start with a concept of combinatorial nature. For any $N\in \NN$, define the function $g_N^{\sharp}$ between two subsets of $D^N$ satisfying the following conditions:
\begin{enumerate}[(i)]
\item $g_N^{\sharp}$ is a bijection between its domain and range;
\item if $g_N^{\sharp}((a_1,\ldots,a_N))=(b_1,\ldots,b_N)$, then
\begin{displaymath}
\sum_{n=1}^N b_n = (d_2-d_1) + \sum_{n=1}^N a_n.
\end{displaymath}
\end{enumerate}
It is easy to see that in general, we cannot define $g_N^{\sharp}$ on the whole set $D^N$. However, as the following lemma points it out, it can be defined on a considerably large subset.

\begin{lemma}\label{lemma_2} The map $g_N^{\sharp}$ can be defined such that
\begin{displaymath}
\# \dom g_N^{\sharp} = k^N(1-o(1)),
\end{displaymath}
as $N\to \infty$.
\end{lemma}
\begin{proof} First of all, split up the set $D^N$ as follows. Take any $0\leq l\leq N$, and any numbers $1\leq c_1<\ldots<c_l\leq N$. Set then $\mathbf{c}=\{c_1,\ldots,c_l\}$ and $\mathbf{c}'=\{1,\ldots,N\}\setminus\{c_1,\ldots,c_l\}$. Further, let $\mathbf{s}: \mathbf{c}' \to D\setminus\{d_1,d_2\}$. Attached to this data, set
\begin{displaymath}
D_{l,\mathbf{c},\mathbf{s}}^N=\{(a_1,\ldots,a_N)\in D^N \text{ such that } \forall c_j\in\mathbf{c}: a_{c_j} \in\{d_1,d_2\} \text{ and } \forall c' \in \mathbf{c}': a_{c'}=\mathbf{s}(c')\},
\end{displaymath}
i.e. $D_{l,\mathbf{c},\mathbf{s}}^N$ stands for those sequences which contain $d_1$'s and $d_2$'s in positions indexed by $\mathbf{c}$, while outside of $\mathbf{c}$, there is a fixed sequence $\mathbf{s}$ made of coefficients other than $d_1,d_2$.

Decompose $D_{l,\mathbf{c},\mathbf{s}}^N$ as
\begin{displaymath}
D_{l,\mathbf{c},\mathbf{s}}^N = \bigcup_{l_1=0}^l D_{l,l_1,\mathbf{c},\mathbf{s}}^N,
\end{displaymath}
where $D_{l,l_1,\mathbf{c},\mathbf{s}}^N$ is the subset of $D_{l,\mathbf{c},\mathbf{s}}^N$ which consists of sequences containing exactly $l_1$ many $d_1$'s. Obviously, $g_N^{\sharp}$ can be defined on a set of size
\begin{displaymath}
\#D^N_{l,\mathbf{c},\mathbf{s}} - \sum_{l_1=0}^l \max(0, \# D_{l,l_1,\mathbf{c},\mathbf{s}}^N- \#D_{l,l_1-1,\mathbf{c},\mathbf{s}}^N),
\end{displaymath}
namely, $g_N^{\sharp}$ maps a sequence in $D^N_{l,\mathbf{c},\mathbf{s}}$ to another one which contains one more copy of $d_2$ and one less copy of $d_1$. Clearly $\#D^N_{l,\mathbf{c},\mathbf{s}}=2^l$ and $\#D^N_{l,l_1,\mathbf{c},\mathbf{s}}= \binom{l}{l_1}$.
We claim that, for any fixed $\eps>0$,
\begin{equation}\label{eq_4}
\sum_{l_1=0}^l \max(0, \# D_{l,l_1,\mathbf{c},\mathbf{s}}^N- \#D_{l,l_1-1,\mathbf{c},\mathbf{s}}^N) \leq 2\eps 2^l + o(2^l),
\end{equation}
as $l\to \infty$. Split this summation up according to $l_1\geq l(1/2-\eps)$ and $l_1< l(1/2-\eps)$. As for the latter, even
\begin{displaymath}
\sum_{l_1< l(1/2-\eps)} \# D^N_{l,l_1,\mathbf{c},\mathbf{s}} = o(2^l),
\end{displaymath}
as $l\to \infty$, following simply from Chebyshev's inequality applied to the random walk of length $l$. Therefore, apart from $o(2^l)$ sequences in $D^N_{l,\mathbf{c},\mathbf{s}}$, we have $l_1\geq l(1/2-\eps)$. In this part of the summation,
\begin{displaymath}
\sum_{l_1\geq l(1/2-\eps)} \max(0, \# D_{l,l_1,\mathbf{c},\mathbf{s}}^N- \#D_{l,l_1-1,\mathbf{c},\mathbf{s}}^N) \leq 2^l \max_{ l(1/2-\eps)\leq l_1 \leq l/2} \left(1-\frac{\#D_{l,l_1-1,\mathbf{c},\mathbf{s}}^N}{\#D_{l,l_1,\mathbf{c},\mathbf{s}}^N}\right) \leq 2^l(2\eps+o(1)),
\end{displaymath}
hence \eqref{eq_4} is established.

It is easy to see that for any fixed $L$, as $N\to\infty$,
\begin{displaymath}
\sum_{l\leq L,\mathbf{c},\mathbf{s}}\# D^N_{l,\mathbf{c},\mathbf{s}} = o(k^N),\qquad \sum_{l > L,\mathbf{c},\mathbf{s}}\# D^N_{l,\mathbf{c},\mathbf{s}} = k^N - o(k^N).
\end{displaymath}
Now let $\delta>0$ be arbitrary. Choose $L$ such that \eqref{eq_4} can be continued as  $2\eps 2^l + o(2^l) < 3\eps 2^l$ for any $N\geq l>L$. Then, with this fixed $L$, if $N$ is large enough, at least $(1-\delta)k^N$ sequences in $D^N$ satisfies $l>L$ (with $l$ standing for the total number of $d_1$'s and $d_2$'s). This altogether yields that $g_N^{\sharp}$ can be defined on a set of size at least $k^N(1-\delta)(1-3\eps)$. Since $\delta>0$ and $\eps>0$ are arbitrary, this completes the proof.
\end{proof}

\begin{remark}\label{remark_1} Similarly we can define the functions $g_N^{\flat}$ which map $(a_1,\ldots,a_N)$ to $(b_1,\ldots,b_N)$ such that
\begin{displaymath}
\sum_{n=1}^N b_n = (d_1-d_2) + \sum_{n=1}^N a_n,
\end{displaymath}
and $g_N^{\flat}$ are bijections between their domain and range. The same argument as that in the proof of Lemma~\ref{lemma_2} gives
\begin{displaymath}
\# \dom g_N^{\flat} = k^N(1-o(1)),
\end{displaymath}
as $N\to \infty$, for a well-chosen function $g_N^{\flat}$.
\end{remark}

From now on, fix two sequences of such functions $g_N^{\sharp}$ and $g_N^{\flat}$ (with domains of size $k^N(1-o(1))$).

\begin{lemma}\label{lemma_3} Both $\mu_+$ and $\mu_-$ are invariant under translations by $d_2-d_1$, i.e. for any Borel set $B\subseteq \RR$,
\begin{displaymath}
\mu_+(B+d_2-d_1)=\mu_+(B),\qquad \mu_-(B+d_2-d_1)=\mu_-(B).
\end{displaymath}
\end{lemma}
\begin{proof} Let $\mu=\mu_+$, the argument for $\mu_-$ is literally the same, writing $\liminf$'s in place of $\limsup$'s. Fix first $\eps>0$.

Define the function $L$ for $S\subseteq \RR$ as follows:
\begin{displaymath}
L(S)=\left\{(a_n)\in\Omega: \limsup_{x\to 1-} \sum_{n=1}^{\infty} a_nx^n \in S\right\}.
\end{displaymath}

On certain sequences $(a_n) \in \Omega$, apply the following sequence of operations: if $(a_1,\ldots,a_N) \in \dom g_N^{\sharp}$, then let
\begin{displaymath}
G_N^{\sharp}((a_n))=(b_n),
\end{displaymath}
where $g_N^{\sharp}((a_1,\ldots,a_N))=(b_1,\ldots,b_N)$, and $b_n=a_n$ for all $n>N$. Now observe that
\begin{displaymath}
\begin{split}
\limsup_{x\to 1-} \sum_{n=1}^{\infty} b_nx^n & = \sum_{n=1}^N b_n + \limsup_{x\to 1-} \sum_{n=N+1}^{\infty} b_nx^n \\ & = d_2-d_1+ \sum_{n=1}^N a_n + \limsup_{x\to 1-} \sum_{n=N+1}^{\infty} a_nx^n = d_2-d_1+ \limsup_{x\to 1-} \sum_{n=1}^{\infty} a_nx^n.
\end{split}
\end{displaymath}
This altogether means that if $(a_n)\in \dom G_N^{\sharp}\cap L(B)$, then $G_N^{\sharp}((a_n))\in L(B+d_2-d_1)$ for any Borel set $B\subseteq \RR$.

By Lemma~\ref{lemma_2}, if $N$ is large enough, $\p(\dom G_N^{\sharp})\geq 1-\eps$, implying $\p(\dom G_N^{\sharp} \cap L(B)) \geq \mu(B)-\eps$. Then, using the simple fact that $G_N^{\sharp}$ preserves $\p$ on its domain, we see
\begin{displaymath}
\mu(B+d_2-d_1)=\p(L(B+d_2-d_1)) \geq \p(G_N^{\sharp}(\dom G_N^{\sharp}\cap L(B))) \geq \mu(B)-\eps.
\end{displaymath}
Since this holds for all $\eps>0$, we obtain
\begin{displaymath}
\mu(B+d_2-d_1)\geq \mu(B).
\end{displaymath}

The same way we obtain $\mu(B-d_2+d_1)\geq \mu(B)$ (see also Remark~\ref{remark_1}), which yields the statement.
\end{proof}

It is well-known (and simple) that there are no nontrivial finite Borel measures on $\RR$ which are invariant under a nontrivial translation, implying $\mu_+(\RR)=\mu_-(\RR)=0$. As it was mentioned in the introduction of this section, this completes the proof of Proposition~\ref{proposition}.

\section{The case of non-vanishing expected value}\label{sec:non-vanishing}

In this section, we prove Theorem~\ref{theorem_1}. Since the two propositions of the theorem are symmetric, we assume $\sum_{j=1}^k d_j > 0$ for the rest of this section.

\begin{lemma} There exists some $K\in \RR$ such that with positive probability, $\sum_{n=1}^{\infty} a_nx^n > K$ holds for any $0<x<1$.
\end{lemma}

\begin{proof} If $\min D\geq 0$, then $K=0$ obviously does the job (the probability in question is just $1$), so assume $\min D<0$ from now on. Set $S_l=\sum_{n=1}^l a_n$ for the partial sums of $\sum_{n=1}^{\infty} a_n$. Then $S_l$ is the sum of $l$ independent and identically distributed random variables. Since $\sum_{j=1}^k d_j > 0$, the expected value of such a random variable is positive. Consequently, by the strong law of large numbers, 
\begin{displaymath} 
\p\left(\sum_{n=1}^{\infty} a_n=\infty\right)=1. 
\end{displaymath}
Using the notation $A_m=\{(a_n)\in\Omega: \sum_{n=1}^{l} a_n > 0 \text{ for any } l>m\}$, this implies, in particular,
\begin{displaymath}
\p\left(\bigcup_{m=1}^{\infty}A_m\right)=1.
\end{displaymath}
This yields that there exists some $m\in \NN$ satisfying $\p(A_m)>0$. Fixing such an $m$, we have
\begin{displaymath}
\p(S_l>m \cdot \min D  \text{ for any } l)>0.
\end{displaymath}

For any $0<x<1$, it is immediate that both $\sum_{n=1}^{\infty} a_nx^n$ and $\sum_{n=1}^{\infty} S_nx^n$ are absolutely convergent, since $a_n\ll_D 1$\footnote{Here, we apply Vinogradov's notation: $A\ll B$ means $|A|\leq cB$ for some constant $c$, while $D$ in the subscript means that this constant $c$ depends only on $D$.} and $S_n\ll_D n$. Then, on the set $\{S_l>m \cdot \min D  \text{ for any } l\}$, by partial summation,
\begin{displaymath}
\sum_{n=1}^{\infty} a_nx^n= \sum_{n=1}^{\infty} (S_n-S_{n-1})x^n = \sum_{n=1}^{\infty} S_n(x^n-x^{n+1})> m \cdot \min D \sum_{n=1}^{\infty} (x^n-x^{n+1})= m\cdot \min D \cdot x.
\end{displaymath}
Therefore, $K=m \cdot \min D$ is an appropriate choice for $K$ in the statement, the set in question is $\{S_l>m \cdot \min D  \text{ for any } l\}$, which is above shown to have positive probability.
\end{proof}

This implies, in particular, that
\begin{displaymath}
\p(\text{$f$ satisfies \eqref{eq_2} or \eqref{eq_3}})<1,
\end{displaymath}
therefore, by Proposition~\ref{proposition},
\begin{equation}\label{eq_5}
\p(\text{$f$ satisfies \eqref{eq_1}})>0.
\end{equation}
Now we finish the proof of Theorem~\ref{theorem_1} by a standard application of Kolmogorov's 0-1 law. Since it will be used once more in the next section, we formulate it as a lemma.

\begin{lemma}\label{lemma_5}
We have
\begin{displaymath}
\p(\text{$f$ satisfies \eqref{eq_1}}), \p(\text{$f$ satisfies \eqref{eq_2}}) \in \{0,1\}.
\end{displaymath}
\end{lemma}
\begin{proof} It is easy to see that the events
\begin{displaymath}
A_{\pm}=\left\{(a_n)\in\Omega: \lim_{x\to 1-} \sum_{n=1}^{\infty} a_nx^n=\pm \infty\right\}
\end{displaymath}
are tail events in the sense of \cite[Section 16.3]{L}, since
\begin{displaymath}
\lim_{x\to 1-} \sum_{n=1}^{\infty} a_nx^n=\pm \infty \text{ if and only if } \lim_{x\to 1-} \sum_{n=N+1}^{\infty} a_nx^n=\pm \infty
\end{displaymath}
holds for any $N\in \NN$, implying that the events $A_{\pm}$ are independent of the first few coefficients $a_1,\ldots,a_N$. Tail events have probability $0$ or $1$ by Kolmogorov's 0-1 law, see \cite[Theorem 16.3 B]{L}.
\end{proof}

Now combining \eqref{eq_5} with Lemma~\ref{lemma_5}, the proof of Theorem~\ref{theorem_1} is complete.

\section{The case of vanishing expected value}

In this section, we are going to prove Theorem~\ref{theorem_2}, so assume $\sum_{j=1}^k d_j=0$. We introduce the following permutation $p$ on $D$: $p(d_j)=d_{j+1}$ for $1\leq j \leq k-1$, and $p(d_k)=d_1$. This gives rise to a permutation $\pp$ on $\Omega$: $\pp((a_n))=(b_n)$, where $b_n=p(a_n)$ for each $n\in\NN$. Now for any $0<x<1$, by absolute convergence,
\begin{displaymath}
\sum_{j=0}^{k-1} f_{\pp^j((a_n))}(x) = \sum_{j=0}^{k-1} \sum_{n=1}^{\infty} p^j(a_n)x^n = \sum_{n=1}^{\infty} \left( \sum_{j=0}^{k-1} p^j(a_n) \right) x^n = \sum_{n=1}^{\infty} \left( \sum_{j=1}^k d_j \right) x^n =0.
\end{displaymath}
Consequently, for any $(a_n)\in\Omega$ and any $0<x<1$, among $f_{(a_n)}(x),f_{\pp((a_n))}(x),\ldots,f_{\pp^{k-1}((a_n))}(x)$ there is at least one nonnegative and at least one nonpositive number. In other words, for any $(a_n)\in\Omega$, as $x\to 1-$, at least one of $f_{(a_n)}(x),f_{\pp((a_n))}(x),\ldots,f_{\pp^{k-1}((a_n))}(x)$ violates \eqref{eq_1} (that one which is nonpositive for some $x$'s arbitrarily close to $1$) and at least one violates \eqref{eq_2} (that one which is nonnegative for some $x$'s arbitrarily close to $1$).

Also, it is easy to see that $\pp$ is $\p$-preserving, altogether yielding
\begin{displaymath}
\p(\text{$f$ satisfies \eqref{eq_1}}), \p(\text{$f$ satisfies \eqref{eq_2}}) \leq 1-1/k.
\end{displaymath}
This, combined with Lemma~\ref{lemma_5}, gives
\begin{displaymath}
\p(\text{$f$ satisfies \eqref{eq_1}}), \p(\text{$f$ satisfies \eqref{eq_2}}) =0.
\end{displaymath}
Now Theorem~\ref{theorem_2} follows from Proposition~{\ref{proposition}.

\section{About residuality}\label{sec:residuality}

In this section, we prove Theorem~\ref{theorem_3}. First assume that each element of $D$ is nonnegative. In this case, we have $\lim_{x\to{1-}}f(x)\neq \infty$ if and only if the sequence of the coefficients contains only finitely many nonzero elements. However, the set $E$ of these sequences is of first category. Indeed, write $E=\bigcup_{m=1}^{\infty}E_m$ where $E_m$ denotes the set of sequences for which $a_n=0$ holds for $n\geq m$. It suffices to see that $E_m$ is nowhere dense in $\Omega$ (for each $m\in\NN$). Given any nonempty open set $U$, it has a nonempty open subset
\begin{displaymath}
V=\{(a_n)\in\Omega \mid a_1=b_1, a_2=b_2, \ldots, a_j=b_j\},
\end{displaymath}
where $j\in\NN$ and $b_1,\ldots,b_j\in{D}$. Now define
\begin{displaymath}
W= \{(a_n)\in\Omega \mid a_1=b_1, a_2=b_2, \ldots, a_j=b_j, a_{\max(j,m)+1} = b\},
\end{displaymath}
where we choose $b\in D$ to be nonzero. Clearly $W\subseteq U$ is nonempty, open, and $W\cap E_m=\emptyset$, therefore the proof of the first statement is complete (the case when each element of $D$ is nonpositive follows by symmetry).

Now let us consider the case in which $D$ contains positive and negative elements simultaneously. One can easily see that $\limsup_{x\to{1-}}f(x)=\infty$ holds if and only if $\sup_{x\in(0,1)}f(x)=\infty$ and $\liminf_{x\to{1-}}f(x)=-\infty$ holds if and only if $\inf_{x\in(0,1)}f(x)=-\infty$. Thus it suffices to prove that $\sup_{x\in(0,1)}f(x)\neq +\infty$ or $\inf_{x\in(0,1)}f(x)\neq -\infty$ hold only in a set of first category. By symmetry, we can focus on the set $F$ where $\sup_{x\in(0,1)}f(x)\neq \infty$ holds. Write it as a countable union $F=\bigcup_{n=1}^{\infty}F_m$ where $F_m$ contains the sequences for which $\sup_{x\in(0,1)}f(x)\leq{} m$. It suffices to see that $F_m$ is nowhere dense in $\Omega$ (for each $m\in\NN$).

Given any nonempty open set $U$, it has a nonempty open subset
\begin{displaymath}
V=\{(a_n)\in\Omega \mid a_1=b_1, a_2=b_2, \ldots, a_j=b_j\},
\end{displaymath}
where $j\in\NN$ and $b_1,\ldots,b_j\in{D}$. Set
\begin{displaymath}
R=\inf_{x\in(0,1)} \sum_{n=1}^{j}{b_nx^n}.
\end{displaymath}
Choose an integer $M>j$ satisfying also $M>(m+1-R)/(\max D) + j$, then
\begin{displaymath}
R+ \max D \sum_{n=j+1}^M 1 > m+1.
\end{displaymath}
Now fix $x<1$ close enough to $1$ such that
\begin{displaymath}
R+\max D \sum_{n=j+1}^M x^n > m+1.
\end{displaymath}
Then choose $N>M$ large enough such that $|\min D \sum_{n=N+1}^{\infty} x^n| < 1$. Taking
\begin{displaymath}
W= \{(a_n)\in\Omega \mid a_1=b_1, a_2=b_2, \ldots, a_j=b_j, a_{j+1}=\ldots=a_N = \max D\},
\end{displaymath}
we have, for $(a_n)\in W$,
\begin{displaymath}
\begin{split}
\sum_{n=1}^{\infty} a_n x^n & \geq \sum_{n=1}^j b_n x^n + \sum_{n=j+1}^N \max D \cdot x^n + \sum_{n=N+1}^{\infty} \min D \cdot x^n \\ & \geq R + \max D \sum_{i=j+1}^N x^n + \min D \sum_{n=N+1}^{\infty} x^n > m+1-1 = m.
\end{split}
\end{displaymath}
Therefore, $W\cap F_m=\emptyset$, and since $W\subseteq U$ is nonempty and open, the proof of Theorem~\ref{theorem_3} is complete. 

\section{Concluding remarks}

It would be interesting to investigate the question of non-uniform distributions, i.e. when $D=\{d_1,\ldots,d_k\}$ and the positive numbers $p_1,\ldots,p_k$ are given such that $\sum_{j=1}^k p_j=1$, and each coefficient takes the value $d_j$ with probability $p_j$ for $1\leq j\leq k$.

Proposition~\ref{proposition} can be proved similarly, apart from the following subtlety. Assuming $p_1\leq p_2$, take the function $g_N^{\sharp}$ with the same properties as above. Then the resulting function $G_N^{\sharp}$ in the proof of Lemma~\ref{lemma_3} does not preserve $\p$ for $p_2>p_1$, but increases it. (Similarly, $G_N^{\flat}$ is $\p$-decreasing.) All in all, although our measure $\mu=\mu_{\pm}$ will not be invariant any more under the translation by $d_2-d_1$, we still have
\begin{displaymath}
\mu(B+d_2-d_1)\geq \mu(B)
\end{displaymath}
for any Borel set $B\subseteq\RR$, and there is no such finite measure on $\RR$ other than the trivial one.

As for Theorem~\ref{theorem_1}, its proof is literally the same, the argument in Section~\ref{sec:non-vanishing} nowhere uses that the coefficients are chosen through a uniform distribution.

However, the statement of Theorem~\ref{theorem_2} for non-uniform distributions remains an open question, we do not see any obvious modification of the argument that would work for general distributions.

\end{document}